\documentclass{amsart}

\usepackage{amsmath}
\usepackage{amssymb}
\usepackage{amscd}
\usepackage[integrals]{wasysym}
\usepackage{hyperref}
\input{Counting2.sty}

\newtheorem{theorem}{Theorem}[section]
\newtheorem{lemma}[theorem]{Lemma}

\newtheorem{corollary}[theorem]{Corollary}
\newtheorem{proposition}[theorem]{Proposition}

\theoremstyle{definition}
\newtheorem{definition}[theorem]{Definition}
\newtheorem{example}[theorem]{Example}

\theoremstyle{remark}
\newtheorem{remark}[theorem]{Remark}
\newtheorem{question}[theorem]{Question}

\newtheorem{conjecture}[theorem]{Conjecture}

\numberwithin{equation}{section}

\let\oldintop\intop
\def\oldint{\oldintop\nolimits}
\newcommand{\GL}{\operatorname{GL}}
\newcommand{\Gr}{\operatorname{Gr}}
\newcommand{\Fl}{\operatorname{Fl}}
\newcommand{\In}{\operatorname{In}}

\newcommand{\Frep}{\operatorname{Frep}}
\newcommand{\module}{\operatorname{mod}}
\newcommand{\op}[1]{\operatorname{#1}}

\newcommand{\Rep}{\operatorname{Rep}}
\newcommand{\Mod}{\operatorname{Mod}}
\newcommand{\Hom}{\operatorname{Hom}}

\newcommand{\Mon}{\operatorname{Mon}}
\newcommand{\Epi}{\operatorname{Epi}}
\newcommand{\Aut}{\operatorname{Aut}}

\newcommand{\Ext}{\operatorname{Ext}}
\newcommand{\ext}{\operatorname{ext}}

\newcommand{\Exp}{\operatorname{Exp}}

\newcommand{\Ker}{\operatorname{Ker}}

\newcommand{\Img}{\operatorname{Im}}

\newcommand{\T}{\operatorname{T}}

\newcommand{\Id}{\operatorname{Id}}

\newcommand{\mc}[1]{\mathcal{#1}}
\newcommand{\mb}[1]{\mathbb{#1}}
\newcommand{\Spec}{\operatorname{Spec}}

\newcommand{\br}[1]{\overline{#1}}
\newcommand{\innerprod}[1]{\langle#1\rangle}
\renewcommand{\ss}[2]{^{{#1}\cdot{\rm #2}}}
\newcommand{\sm}[1]{\left(\begin{smallmatrix}#1\end{smallmatrix}\right)}
\newcommand{\smvar}[1]{\bigl[\begin{smallmatrix}#1\end{smallmatrix}\bigr]}

\def\Xint#1{\mathchoice
{\XXint\displaystyle\textstyle{#1}}%
{\XXint\textstyle\scriptstyle{#1}}%
{\XXint\scriptstyle\scriptscriptstyle{#1}}%
{\XXint\scriptscriptstyle\scriptscriptstyle{#1}}%
\!\oldint}
\def\XXint#1#2#3{{\setbox0=\hbox{$#1{#2#3}{\oldint}$}
\vcenter{\hbox{$#2#3$}}\kern-.5\wd0}}

\makeatletter
\def\equalsfill{$\m@th\mathord=\mkern-7mu
\cleaders\hbox{$\!\mathord=\!$}\hfill
\mkern-7mu\mathord=$}
\makeatother

\begin{document}

\title{Counting using Hall Algebras II. Extensions from Quivers}

\author{Jiarui Fei}
\address{Department of Mathematics, University of California, Riverside, CA 92521, USA}
\email{jiarui@ucr.edu}
\thanks{}


\date{}
\dedicatory{}

\begin{abstract} We count the $\mb{F}_q$-rational points of GIT quotients of quiver representations with relations. We focus on two types of algebras -- one is one-point extended from a quiver $Q$, and the other is the Dynkin $A_2$ tensored with $Q$. For both, we obtain explicit formulas.
We study when they are polynomial-count. We follow the similar line as in the first paper but algebraic manipulations in Ringel-Hall algebra will be replaced by corresponding geometric constructions.
\end{abstract}

\maketitle

\section*{Introduction}

We continue our development on algorithms to count the points of various representation varieties of a quiver with relations.
In this note, we will mainly focus on a class of algebras called one-point extensions from a quiver.
Unless otherwise specified our base field $k$ is the finite field $\mb{F}_q$ with $q$ elements.
Let $Q$ be any finite quiver and $E$ a representation of $Q$. The one-point extension of $Q$ by $E$ is the triangular algebra
$kQ[E]:=\sm{kQ\ 0\\\,E\:\ k}$. We also interested in the tensor product algebra
$kA_2(Q):=kA_2\otimes kQ$, where $A_2$ is the Dynkin quiver of type $A_2$.
Such algebras include a large class of {\em truncated Jacobian algebras}. The results established here can be applied to the quantum cluster algebra theory \cite{Fc3}.

In \cite{Fc1}, we applied serval counting characters to the Harder-Narasimhan identity \eqref{eq:HallID} in the Ringel-Hall algebra of a quiver and obtained several interesting formulas. All characters that we considered are originated from Reineke's counting character $\oldint$ from the Hall algebra to certain quantum power series ring.
Unfortunately $\oldint$ fails to be an algebra morphism for non-hereditary algebras, though Harder-Narasimhan identity exists quite generally. However, applying the same map $\oldint$ to the HN-identity can still generate effective counting formulas.
We will follow the similar line as the first paper.
The only change is that we replace algebraic manipulations in the Hall algebras by corresponding geometric constructions.

We first state the main results of this notes. Let $A$ be any basic algebra presented by $A=kQ/I$. Fix a slope function $\mu$, and we denote by $\Rep_\alpha^\mu(A)$ the variety of $\alpha$-dimensional $\mu$-semistable representations of $A$, and by $\Mod_\alpha^\mu(A)$ its GIT quotient.

\begin{lemma}  $|\Rep_\alpha^\mu(A)|=\sum (-1)^{s-1} |\Frep_{\alpha_1\cdots\alpha_s}(A)|,$ where the sum runs over all decomposition $\alpha_1+\cdots+\alpha_s=\alpha$ of $\alpha$ into non-zero dimension vectors such that $\mu(\sum_{l=1}^k\alpha_l)<\mu(\alpha)$ for $k<s$.
\end{lemma}
We will define the key varieties $\Frep_{\alpha_1\cdots\alpha_s}(A)$ in Section \ref{S:HN}. In particular, if all $\Frep$ varieties can be effectively counted, then so are $\Rep_\alpha^\mu(A)$. The map $\oldint$ have so-called $\Delta$ and $S$ analogs. They are defined in \cite{Fc1} as $\oldint_\Delta$ and $\Xint S$. Here, $\Delta$ and $S$ are related to the comultiplication and the antipode \cite{X} in the Hall algebra.
In the current setting, Lemma 0.1 and counting formula for $\Frep$ varieties have $\Delta$ and $S$ analogs as well.

\begin{lemma}
For $A=kQ[E]$ or $kA_2(Q)$, we have explicit counting formulas for $\Frep$ varieties, and those formulas have $\Delta$ and $S$ analogs.
\end{lemma}

Recall that a variety
$X$ is called {\em polynomial-count} (or has a counting polynomial) if there exists a (necessarily unique) polynomial $P_X=\sum a_i t_i\in\mb{C}[t]$ such that for every finite extension $\mb{F}_{q^r}/\mb{F}_q$, we have $|X(\mb{F}_{q^r})|=P_X(q^r)$.
We are especially interested in when all $\Frep$ varieties are polynomial-count.
If this is the case, it is clear that each $\Mod_\alpha^\mu(A)$ is polynomial-count when it is a geometric quotient.

\begin{theorem}{\ } \label{T:intro}
\begin{enumerate}
\item For $A=kQ[E]$, all GIT quotients $\Mod_\alpha^\mu(A)$ can be explicitly counted in terms of quiver Grassmannians of $E$. If $E$ is {\em add-polynomial-count}, then all $\Mod_\alpha^\mu(A)$ are polynomial-count.
\item For $A=kA_2(Q)$, all $\Mod_\alpha^\mu(A)$ have counting polynomials, which can be explicitly computed.
\item If $E$ is add-polynomial-count, $\Mod_\alpha^\mu(kA_2\otimes kQ[E])$ is polynomial-count for certain choice of $\alpha$ and $\mu$.
\end{enumerate}
\end{theorem}

This notes are organized as follows.
In Section \ref{S:Prelim}, we provide necessary background on the representation theory of quivers with relations and points counting.
In Section \ref{S:HN}, we introduce the $\Frep$ variety and the notion of F-polynomial-count. After recalling the Harder-Narasimhan identity in the Hall algebra, we conclude our key lemma (Lemma \ref{L:Tao}).
In Section \ref{S:ext}, we first review the trivial extension of algebras in general, then specialize to the case of one-point extensions from a quiver. We describe the relations of these algebras from the projective presentation of $E$.
In Section \ref{S:Frep}, we show in Lemma \ref{L:Frep} that their $\Frep$ varieties can be counted in terms of the usual representation varieties.
In Section \ref{S:rep}, we show in Lemma \ref{L:rep} that these usual representation varieties can be counted in terms of the Grassmannians of $nE$. Motivated by this, we introduce add-polynomial-count property for a representation. We conclude by our first main result -- Theorem \ref{T:ext} (Theorem \ref{T:intro}.(1)).
Many examples will follow in Section \ref{S:example}.
In Section \ref{S:HS}, we apply our algorithm to count homological strata on the geometric quotients. The method is outlined in Theorem \ref{T:HS}.
In Section \ref{S:A2}, we work with the algebra $kA_2(Q)$. Our second main result -- Theorem \ref{T:A2Q} (Theorem \ref{T:intro}.(2)) gives an analogous counting formula, which is independent of Grassmannians of representations.
In Section \ref{S:delta}, we consider the $\Delta$-analog of counting. We introduce the $\Delta$-analog of the $\Frep$ varieties. Lemma \ref{L:Frep2} is the $\Delta$-analog of Lemma \ref{L:Frep}. We conclude by our third main results -- Theorem \ref{T:A2ext} (Theorem \ref{T:intro}.(3)).
Finally in Section \ref{S:S}, we consider the $S$-analog of counting. Our final main results is Theorem \ref{T:final}, which removes the assumption of being a geometric quotient in the previous results.

Most of our constructions can be easily generalized to the motivic setting. Since the main application of this theory is in the quantum algebra, we shall not pursue that generality. The geometry of these moduli spaces will be studied in another series of notes \cite{Fm2}.

\section{Preliminary} \label{S:Prelim}
\subsection{Quivers with Relations}
Let $Q$ be a finite quiver with the set of vertices $Q_0$ and the set of arrows $Q_1$. If $a\in Q_1$ is an arrow, then $ta$ and $ha$ denote its tail and its head respectively. Fix a {\em dimension vector} $\alpha$, the space of all $\alpha$-dimensional representations over a field $k$ is
$$\Rep_\alpha(Q):=\bigoplus_{a\in Q_1}\Hom(k^{\alpha(ta)},k^{\alpha(ha)}).$$
The group $G=\GL_\alpha:=\prod_{v\in Q_0}\GL_{\alpha(v)}$ acts on $\Rep_\alpha(Q)$ by the natural base change. Two representations $M,N\in\Rep_\alpha(Q)$ are isomorphic if they lie in the same $\GL_\alpha$-orbit.

Let $kQ$ be the {\em path algebra} of $Q$ over $k$, then $M\in\Rep_\alpha(Q)$ is naturally a (right) $kQ$-module.
Fix a set $R$ of homogeneous elements in $kQ$ with respect to the bigrading:
$kQ=\bigoplus_{u,v\in Q_0}e_u kQ e_v$. Here, $e_v$ is the trivial path on the vertex $v$.
Due to the homogeneousity, for each $r\in R$ there are $tr,hr\in Q_0$ such that $r$ is a linear combination of paths from $tr$ to $hr$.
If $M(r)=0$ for all $r\in R$, then we say $M$ is a {\em representation of $Q$ with relations $R$}.
The {\em path algebra of $Q$ with relations $R$} is the quotient algebra $A:=kQ/\innerprod{R}$. A representation of $Q$ with relations $R$ naturally becomes an $A$-module.

The assignment $M\mapsto M(r)$ defines a polynomial map
$ev(r):\Rep_{\alpha}(Q)\to \Hom(k^{\alpha(tr)},k^{\alpha(hr)}),$
which is represented by an $\alpha(hr)\times\alpha(tr)$ matrix with entries in $k[\Rep_{\alpha}(Q)]$.
Let $\tilde{R}\subseteq k[\Rep_{\alpha}(Q)]$ be the ideal generated by the entries of all $ev(r)$ for which  $r\in R$. The representation space $\Rep_{\alpha}(A)$ is the scheme $\Spec(k[\Rep_{\alpha}(Q)]/\tilde{R})$. As a variety, $\Rep_{\alpha}(A)$ consists of all $\alpha$-dimensional representations of $A$.


\subsection{Stability and GIT Quotients}
A {\em weight} $\sigma$ is an integral linear functional on $\mathbb{Z}^{Q_0}$. A {\em slope function} $\mu$ is a quotient of two weights $\sigma/\theta$ with $\theta(\alpha)>0$ for any non-zero dimension vector $\alpha$.
For a representation $M$ of $Q$, we denote by $\br{M}$ the dimension vector of $M$.

\begin{definition} A representation $M$ is called {\em $\mu$-semi-stable (resp. $\mu$-stable)} if
$\mu(\br{L})\leqslant \mu(\br{M})$ (resp. $\mu(\br{L})<\mu(\br{M})$) for every non-trivial subrepresentation $L\subset M$.
\end{definition}

We denote by $\Rep_\alpha^\mu(A)$ the variety of $\alpha$-dimensional $\mu$-semistable representations of $A$.
By the standard GIT construction \cite{Ki}, there is a {\em categorical quotient} $q: \Rep_\alpha^\mu(A)\to \Mod_\alpha^\mu(A)$ and its restriction to the stable representations $\Rep_\alpha\ss{\mu}{st}(A)$ is a {\em geometric quotient}.
In King's paper, the construction is done over an algebraically closed field, but according to \cite{Se} this can also be done over finite fields.

A slope function $\mu$ is called {\em coprime} to $\alpha$ if $\mu(\gamma)\neq \mu(\alpha)$ for any $\gamma<\alpha$. So if $\mu$ is coprime to $\alpha$, then there is no strictly semistable (semistable but not stable) representation of dimension $\alpha$. In this case, $\Mod_\alpha^\mu(A)$ must be a geometric quotient.

Note that the semi-stable objects with a fixed slope $\mu_0$ form an exact subcategory $\module_{\mu_0}(A)$.
For any dimension vector $\alpha$, we can always modify $\mu$ to get a new slope function $\mu_\alpha$ such that $\mu_\alpha(\alpha)=0$ and $\Rep_\alpha^{\mu_\alpha}(A)=\Rep_\alpha^\mu(A)$.
If $\mu(\alpha)=\frac{\sigma(\alpha)}{\theta(\alpha)}$, then we can take $\mu_\alpha=\frac{\sigma_\alpha}{\theta}$, where $\sigma_\alpha=\theta(\alpha)\sigma-\sigma(\alpha)\theta$.

\begin{lemma} {\em \cite[Proposition 2.5]{R1}} {\em Harder-Narasimhan filtration:}\\
Every representation $M$ has a unique filtration
$$0=M_0\subset M_1\subset\cdots\subset M_{m-1}\subset M_{m}=M$$
such that $N_i=M_i/M_{i+1}$ is $\mu$-semi-stable and $\mu(\br{N}_i)>\mu(\br{N}_{i+1})$.
\end{lemma}

\subsection{Counting and Cohomology}
Let $X$ be a variety over $k=\mb{F}_q$ and $X_{\br{k}}=X\otimes_k\br{k}$.
We denote by $H_c^i(X,\mb{Q}_l)$ the $i$-th $l$-adic cohomology group with compact support of $X_{\br{k}}$ with $l\neq \op{char} k$.
The key method for counting rational points on $X$ is given by the Grothendieck-Lefschetz trace formula:
$$|X(\mb{F}_{q^r})|=\sum_{i=0}^{2\dim X}(-1)^i Tr\big(F^r;H_c^i(X,\mb{Q}_l)\big),$$
where $F$ is the Frobenius morphism $X_{\br{k}}\to X_{\br{k}}$.
$X$ is called {\em $l$-pure} if the eigenvalues of $F$ on $H_c^i(X,\mb{Q}_l)$ have absolute value $q^{i/2}$. It is known that if $X$ is smooth and proper over $\br{k}$ then $X$ is $l$-pure.


\begin{lemma} {\em \cite[Proposition 6.1]{R2}} \label{L:Euler}
If $X$ is counted by a rational function $P_X$, then $P_X$ must lie in $\mb{Z}[t]$. Its specialization at $q=1$ gives the $l$-adic Euler characteristic of $X_{\br{k}}$.
\end{lemma}

\begin{definition} The {\em Poincar\'{e} polynomial} $P(X,q)\in\mb{Z}[q^{1/2}]$ of $X$ is
$$P(X,q)=\sum_{i\geq 0} (-1)^i\dim H_c^i(X,\mb{Q}_l)q^{i/2}.$$
\end{definition}

\begin{lemma} \cite[Lemma A.1]{CV} \label{L:polycount} Assume that $X$ is $l$-pure and polynomial-count. Then $P_X(q)=P(X,q)$.
In particular, $P_X(t)\in\mb{N}[t]$.
\end{lemma}

\section{HN-Filtration Identity} \label{S:HN}

Let $A$ be any basic algebra presented by $A=kQ/I$ for $k$ the finite field $\mb{F}_q$.
For any decomposition of dimension vector $\alpha=\sum_{i=1}^s\alpha_i$, we define $\Fl_{\alpha_s\cdots\alpha_1}:=\prod_{v\in Q_0}\Fl_{\alpha_s(v)\cdots\alpha_1(v)}$, where $\Fl_{\alpha_s(v)\cdots\alpha_1(v)}$ is the usual flag variety parameterizing flags of subspaces of dimension $\alpha_1(v)<\alpha_1(v)+\alpha_2(v)<\cdots<\alpha_1(v)+\cdots+\alpha_{s-1}(v)$ in $k^{\alpha(v)}$. To simplify the notation, we denote $\dot{\alpha}_i:=\sum_{j=1}^i\alpha_j$.

\begin{definition} \label{D:Frep} We define the {\em Frep} variety $\Frep_{\alpha_s\cdots\alpha_1}(A):=$
$$\{(M,L_1,\dots,L_{s-1})\in \Rep_{\alpha}(A) \times\Fl_{\alpha_s\cdots\alpha_1}\mid L_1\subset\cdots\subset L_s=M \text{ are representations} \}.$$
\end{definition}
\noindent Let $r:\Frep_{\alpha_s\cdots\alpha_1}(A)\to\Rep_\alpha(A)$ be the projection, the {\em flag variety} $\Fl_{\alpha_s\cdots\alpha_1}(M)$ {of $M$} is the fibre $r^{-1}(M)$, and its subvarieties
$\Fl_{N_s,\dots,N_1}(M)$ is $$\{(L_1,\dots,L_{s-1})\in\Fl_{\alpha_s\cdots\alpha_1}(M)\mid L_{i}/L_{i-1}\cong N_i\}.$$
When the flag is only 2-step, we may use the usual Grassmannian notation. For example, $\Gr_\gamma(\alpha):=\Fl_{\beta,\gamma}$ and $\Gr_\gamma(M):=\Fl_{\beta,\gamma}(M)$, where $\beta=\alpha-\gamma$.

For any three $A$-modules $U,V$ and $W$ with dimension vector $\beta,\gamma$ and $\alpha=\beta+\gamma$, the {\em Hall number} $F_{UV}^W$ is by definition $|\Fl_{U,V}(W)|$.
We denote $a_W:=|\Aut_Q(W)|.$
Let $H(A)$ be the space of all formal (infinite) linear combinations of isomorphism classes $[M]$ in $A$-$\module$.
\begin{lemma}\cite{R}
The completed {\em Ringel-Hall algebra} $H(A)$ is the associative algebra with multiplication
$$[U][V]:=\sum_{[W]}F_{UV}^W[W],$$ and unit $\eta:k\mapsto k[0]$.
\end{lemma}


We fix a slope function $\mu$. For a dimension vector $\alpha$, let $\chi_\alpha=\sum_{\br{M}=\alpha}[M]$ and $\chi_\alpha^\mu=\sum_{M\in\module_\alpha^\mu(A)}[M]$.
We consider a simple counting map $\int: H(A)\to \mb{Q}(q)$ defined by $[M]\mapsto a_M^{-1}$.
Since $\frac{1}{a_M}=\frac{|\mc{O}_M|}{|\GL_\alpha|}$, we have that $\int \chi_\alpha=\frac{|\Rep_{\alpha}(A)|}{|\GL_{\alpha}|}$. We denote the function $\frac{|\Rep_{\alpha}(A)|}{|\GL_{\alpha}|}$ by $r_\alpha(A,q)$. In general, this function may not be rational in $q$.

The existence of the Harder-Narasimhan filtration yields the following identity in the Hall algebra $H(A)$.
\begin{lemma} \cite[Proposition 4.8]{R1} \label{L:HNid} $$\chi_\alpha=\sum_{} \chi_{\alpha_1}^\mu\cdot\dots\cdot\chi_{\alpha_s}^\mu,$$ where the sum running over all decomposition $\alpha_1+\cdots+\alpha_s=\alpha$ of $\alpha$ into non-zero dimension vectors such that $\mu(\alpha_1)<\cdots<\mu(\alpha_s)$. In particular, solving recursively for $\chi_\alpha^\mu$, we get
\begin{equation}\label{eq:HallID} \chi_\alpha^\mu=\sum_* (-1)^{s-1}\chi_{\alpha_1}\cdot\cdots\chi_{\alpha_s},
\end{equation}
where the sum runs over all decomposition $\alpha_1+\cdots+\alpha_s=\alpha$ of $\alpha$ into non-zero dimension vectors such that $\mu(\sum_{l=1}^k\alpha_l)<\mu(\alpha)$ for $k<s$.
\end{lemma}

Our key observation is that
\begin{equation} \label{eq:Frep} r_{_{\alpha_1\cdots\alpha_s}}(A):=\int \chi_{\alpha_1}\cdots\chi_{\alpha_s}=\frac{|\Frep_{\alpha_1\cdots\alpha_s}(A)|}{|\GL_\alpha|}.
\end{equation}
So the problem boils down to counting these $\Frep$ varieties. In this paper, we will only focus on a class of algebras, for which these varieties can be effectively counted.

\begin{definition} \label{D:FPC}
We say an algebra $A$ is {\em polynomial-count} if each $\Rep_\alpha(A)$ is polynomial-count. It is called {\em F-polynomial-count} if each $\Frep_{\alpha_1\cdots\alpha_s}(A)$ is polynomial-count.
\end{definition}

\noindent We do not know a single example where $A$ is polynomial-count but not F-polynomial-count. We suspect that they are actually equivalent.
It follows from \eqref{eq:HallID} and \eqref{eq:Frep} that

\begin{lemma} \label{L:Tao} $$|\Rep_\alpha^\mu(A)|=\sum_* (-1)^{s-1} |\Frep_{\alpha_1\cdots\alpha_s}(A)|.$$
In particular, if $A$ is F-polynomial-count, then each $\Mod_\alpha^\mu(A)$ is polynomial-count when it is a geometric quotient.
\end{lemma}

\begin{conjecture} The assumption of being a geometric quotient in the lemma can be dropped.
\end{conjecture}

\noindent We shall see in the end that the conjecture is true for the two main classes of algebras considered in this paper.

\section{Trivial Extensions} \label{S:ext}
Given two finite-dimensional $k$-algebras $A,B$ and a $A$-$B$-bimodule $E$, we get the {\em trivial extension algebra} $B[E]=\sm{B & 0\\ _AE_B & A}$.
In the meanwhile, we can form the category $\Rep(_AE_B)$ of representations of the bimodule $_AE_B$ as follows:
The objects are triples $(M_A,M_B,\varphi)\in\module A\times\module B\times\Hom_B(M_A\otimes_A E,M_B)$. A morphism from $(M_A,M_B,\varphi)$ to $(N_A,N_B,\psi)$ is a pair $(f_A,f_B)$ making the following diagram commutes
$$\xymatrix{
M\otimes_A E \ar[r]^{\varphi} \ar[d]_{f_A\otimes\Id} & M_B\ar[d]_{f_B}\\
N\otimes_A E \ar[r]^{\psi} & N_B } $$

\begin{lemma} \cite[A.2.7]{ASS} The two categories $\Rep(B[E])$ and $\Rep(_AE_B)$ are equivalent.
\end{lemma}

\begin{proof} The equivalence is given by $F(M)=(M_A,M_B,\varphi)$, where $M_A=Me_A,M_B=Me_B$ and $\varphi(m\otimes \epsilon)=m\sm{0 & 0\\ \epsilon & 0}e_B$.
\end{proof}

In particular, if $M\in\Rep(B[E])$ corresponds to $(M_A,M_B,\varphi)$, then the dimension vector of $M$ is given by $(\br{M}_A,\br{M}_B)$. One particular case we interested in is when $E$ is simply a right $B$-module. Treating $E$ as a $k$-$B$-bimodule, the triangular algebra $\sm{B & 0\\ E & k}$ is called (trivial) {\em one-point extension} of $B$ by $E$. There is an obvious dual notion of one-point coextension $B[E^*]:=\sm{k & 0\\ E^* & B}$.

\begin{lemma} \label{L:variety} {\ }\begin{enumerate}
\item[$\bullet$] $\Rep_{(n,\alpha)}(B[E])$ is the subvariety of $\Rep_\alpha(B)\times\Hom(nE,k^\alpha)$ defined as
$$\{(M,f)\in \Rep_\alpha(B)\times\Hom(nE,k^\alpha)\mid f\in\Hom_B(nE,M)\}.$$
\item[$\bullet$] $\Rep_{(\alpha,n)}(B[E^*])$ is the subvariety of $\Rep_\alpha(B)\times\Hom(k^\alpha,nE)$ defined as
$$\{(M,f)\in \Rep_\alpha(B)\times\Hom(k^\alpha,nE)\mid f\in\Hom_B(M,nE)\}.$$
\end{enumerate}
\end{lemma}

Suppose that $E\in\Rep(Q)$ is presented by $0\to P_1\xrightarrow{D} P_0\to E\to 0$ with $P_1=\bigoplus_v b_v^1 P_v$ and $P_0=\bigoplus_v b_v^0 P_v$, where $P_v$ is the indecomposable projective representation corresponding to the vertex $v$.
Then the algebra $A=kQ[E]$ can be presented by a new quiver $Q(E)$, which is obtained from $Q$ by adjoining a new vertex ``$-$" and for each $P_v$ in $P_0$ a new arrow from ``$-$" to the vertex $v$. The relations are clearly given by the matrix $D$. In reality, the presentation is always chosen to be minimal. By abuse of notation, we also use $Q[E]$ to denote the new quiver $Q(E)$ with those new relations. The one-point coextension $kQ[E^*]$ can be similarly described using injective presentation of $E$. By convention, the newly adjoined vertex is denoted by ``+".

It is clear that $E$ is the first syzygy of $S_{-}$, so $0\to P_1\xrightarrow{D} P_0\to P_{-} \to S_{-}$.
Moreover, a simple representation of $Q[E]$ is either $S_{-}$ or a simple representation of $Q$.
So we conclude that $kQ[E]$ has global dimension at most $2$. It is easy to compute the matrices $\mc{E}_A^i:=\big(\ext_A^i(S_u,S_v)\big)$. Let $\mc{E}_Q^i:=\big(\ext_Q^i(S_u,S_v)\big)$, then
$\mc{E}_A^0=\sm{1 & 0 \\ 0 & \mc{E}_Q^0}, \mc{E}_A^1=\sm{0 & b^0 \\ 0 & \mc{E}_Q^1}, \mc{E}_A^2=\sm{0 & b^1 \\ 0 & 0},$ so the Euler matrix
$\mc{E}_A$ of $A$ is $\sm{1 & \delta \\ 0 & \mc{E}_Q}$, where $\delta=b^1-b^0$ and $\mc{E}_Q$ is the Euler matrix of $Q$. Throughout this notes, $\innerprod{-,-}$ is the multiplicative Euler form of the quiver $Q$, that is, $\innerprod{\alpha,\beta}=q^{\alpha\mc{E}_Q\beta^{\T}}$. Similarly, we define $\innerprod{\alpha,\beta}_i=q^{\alpha\mc{E}_Q^i\beta^{\T}}$ for $i=0,1$.

\section{Counting $\Frep$ of $Q[E]$} \label{S:Frep}
To simplify our notation, we always use letter with tilde to indicate that $\tilde{M}\in\Rep(Q[E])$ can be represented by $(M,f_M)$, where $M\in\Rep(Q)$ and $f_M:nE\to M$.
A dimension vector with tilde, say $\tilde{\beta}$, consists of two components $(\beta_-,\beta)$, or $(\beta,\beta_+)$ for coextension.
We always set $\tilde{\alpha}=\tilde{\beta}+\tilde{\gamma}$.

\begin{lemma} \label{L:Frep} $p:\Frep_{{\tilde{\beta},\tilde{\gamma}}}(Q[E])\to \Fl_{{\tilde{\beta},\tilde{\gamma}}}$ is a fibre bundle with fibre
$$\Rep_{(\alpha_-,\gamma)}(Q[E])\times \Rep_{\tilde{\beta}}(Q[E])\times \prod_{a\in Q_1}\Hom(k^{\beta(ta)},k^{\gamma(ha)}).$$
So $$r_{\tilde{\beta},\tilde{\gamma}}(Q[E]):=\frac{|\Frep_{{\tilde{\beta},\tilde{\gamma}}}(Q[E])|}{|\GL_{\tilde{\alpha}}|}
=\innerprod{\beta,\gamma}^{-1}\smvar{\alpha_- \\ \gamma_-}|\GL_{\beta_-}|r_{\tilde{\beta}}(Q[E])r_{(\alpha_-,\gamma)}(Q[E]),$$
where $\smvar{n \\ m}$ is the quantum binomial coefficient.
\end{lemma}

\begin{proof} We will sketch the fibre bundle construction by a picture.
After fixing an elements in $\Fl_{{\tilde{\beta},\tilde{\gamma}}}$, we need to fill in the missing part for a $\tilde{\alpha}$-dimensional representation of $Q[E]$. The missing part consists of a $\tilde{\gamma}$-dimensional representation $S$, a $\tilde{\beta}$-dimensional representation $T$, and a bunch of linear maps from $T(ta)$ to $S(ha)$, as indicated below.
$$\Frepext{T}{S}$$
We can stuff the space in the order below independently.
The linear maps from $T_-$ together with all representations $S$ can be identified with $\Rep_{(\alpha_-,\gamma)}(Q[E])$; all representations $T$ can be identified with $\Rep_{\tilde{\beta}}(Q[E])$, and the rest of the linear maps are $\prod_{a\in Q_1}\Hom(k^{\beta(ta)},k^{\gamma(ha)})$.

For the last formula, we only need to notice that
$$|\Fl_{{\tilde{\beta},\tilde{\gamma}}}|=\frac{\smvar{\alpha_-\\\gamma_-}|\GL_{\alpha}|}{|\GL_{\beta}||\GL_{\gamma}|\innerprod{\beta,\gamma}_0}.$$
\end{proof}

The above 2-step case can be recursively generalized to the $n$-step case. We only state the analog for the last formula and its dual.
Here the convention is that $\tilde{\alpha}_0$ is the zero vector.
\begin{align} \label{eq:frext} r_{\tilde{\alpha_1}\cdots\tilde{\alpha}_s}(Q[E])&=\prod_{i=1}^s
\innerprod{\dot{\alpha}_{i-1},\alpha_i}^{-1}\smvar{\dot{\alpha}_{i,-} \\ \alpha_{i,-} }|\GL_{\dot{\alpha}_{i-1,-}}|r_{(\dot{\alpha}_{i,-},\alpha_i)}(Q[E]);\\
r_{\tilde{\alpha}_s\cdots\tilde{\alpha}_1}(Q[E^*])&=\prod_{i=1}^s
\innerprod{\alpha_i,\dot{\alpha}_{i-1}}^{-1}\smvar{\dot{\alpha}_{i,+} \\ \alpha_{i,+}}|\GL_{\dot{\alpha}_{i-1,+}}|r_{(\alpha_i,\dot{\alpha}_{i,+})}(Q[E^*]).
\end{align}
Now the problem boils down to count those affine representation varieties $\Rep_{\tilde{\alpha}}(Q[E])$.

\section{Counting Affine} \label{S:rep}

For any dimension vector $\beta$, we denote by $\Gr^\beta(E)$ the variety parameterizing all $\beta$-dimensional quotient representations of $E$, and define
\begin{align*} \Hom_Q(E,\alpha)_\beta=\{(M,\phi,E_1,M_1&)\in \Rep_\alpha(Q) \times \Hom(E,k^\alpha) \times\Gr^{\beta}(E)\times \Gr_\beta(\alpha) \mid \\
& \phi\in\Hom_Q(E,M), E/\Ker\phi=E_1, \Img\phi=M_1\}.\end{align*}

\begin{lemma} \label{L:rep}
$p:\Hom_Q(E,\alpha)_\beta\to \Gr^{\beta}(E)\times \Gr_\beta(\alpha)$ is a fibre bundle with fibre
$$\GL_\beta\times \Rep_{\alpha-\beta}(Q)\times \bigoplus_{a\in Q_1}(\Hom(k^{(\alpha-\beta)(ta)},k^{\beta(ha)}).$$
So
$$r_{(n,\alpha)}(Q[E])=\sum_{\alpha=\gamma+\beta}\frac{|\Gr^{\beta}(nE)|}{\innerprod{\gamma,\beta}|\GL_n|}r_{\gamma}(Q).$$
Dually, we have a formula for the one-point coextension:
$$r_{(\alpha,n)}(Q[E^*])=\sum_{\alpha=\gamma+\beta}\frac{|\Gr_{\gamma}(nE)|}{\innerprod{\gamma,\beta}|\GL_n|}r_{\beta}(Q).$$
\end{lemma}

\begin{proof} The fibre bundle construction is not hard to verify. Since we only need the last formula, we give a Hall algebra proof for that.
We denote by $\Mon_Q(M,N)$ and $\Epi_Q(M,N)$ the set of all monomorphisms and epimorphisms from $M$ to $N$ respectively. Fix a representation $M$, then the following identities clearly holds in the Hall algebra $H(Q)$
$$\Big(\sum_{[U]}[U]\Big)\Big(\sum_{[V]}|\Epi_Q(M,V)|[V]\Big) = \sum_{[W]}|\Hom_Q(M,W)|[W].$$
Let $\mc{P}_{\innerprod{Q}}$ be the completed quantum polynomial algebra $\mb{Q}(q)[\bold{x}]$, where the multiplication rule is $x^\alpha x^\beta=\innerprod{\alpha,\beta}^{-1}x^{\alpha+\beta}$.
Then the map $\oldint: H(Q)\to \mc{P}_{\innerprod{Q}}$ sending $[M]\to a_M^{-1}x^{\br{M}}$ is an algebra morphism \cite{R1}.
Here, we use the slanted $\int$ to distinguish the one with target $\mb{Q}(q)$.
Apply $\oldint$ to both sides, we get
\begin{align*}  & \oldint \sum_{[U]}[U] \oldint \sum_{[V]}|\Epi_Q(M,V)|[V] = \oldint\sum_{[W]}|\Hom_Q(M,W)|[W] \\
 \Leftrightarrow & \sum_\gamma r_\gamma(Q)x^\gamma\sum_{[V]}a_V^{-1}|\Epi_Q(M,V)|x^{\br{V}} = \sum_{[W]}a_W^{-1}|\Hom_Q(M,W)|x^{\br{W}} \\
 \Leftrightarrow & \sum_\gamma r_\gamma(Q)x^\gamma\sum_\beta |\Gr^\beta(M)| x^\beta= \sum_{[W]}\frac{|\mc{O}_W|}{|\GL_\alpha|}|\Hom_Q(M,W)|x^\alpha \quad (\alpha=\br{W}) \\
 \Leftrightarrow & \sum_{\beta+\gamma=\alpha}\innerprod{\gamma,\beta}^{-1}r_{\gamma}(Q)|\Gr^{\beta}(M)|=\sum_{[W]\mid \br{W}=\alpha}\frac{|\mc{O}_W|}{|\GL_\alpha|}|\Hom_Q(M,W)|.
\end{align*}
Now we set $M:=nE$, then the formula follows from Lemma \ref{L:variety}. The dual formula can be obtained by applying $\oldint$ to the identity:
$$\Big(\sum_{[U]}|\Mon_Q(U,M)|[U]\Big)\cdot\Big(\sum_{[V]}[V]\Big) = \sum_{[V]}|\Hom_Q(W,M)|[W].$$
\end{proof}

\begin{remark} \label{R:gen} Let $R(A)$ be the generating functions $R(A):=\sum_\alpha r_\alpha(A)x^\alpha$, and
$$F^\bullet(E):=\sum_\beta \Gr^\beta(E)x^\beta, \text{\quad and\quad} F_\bullet(E):=\sum_\gamma \Gr_\gamma(E)x^\gamma.$$
If we set
$$F^\infty(E)=\sum_{n=1}^{\infty}\frac{F^\bullet(nE)x_{-}^n}{|\GL_n|} \text{\quad and\quad} F_\infty(E)=\sum_{n=1}^{\infty}\frac{F_\bullet(nE)x_{+}^n}{|\GL_n|},$$
then Lemma \ref{L:rep} can be rewritten as the equations in $\mc{P}_{\innerprod{Q}}[x_{\pm}]$:
$$R(Q[E])=R(Q)F^\infty(E), \text{\quad and\quad} R(Q[E^*])=F_\infty(E)R(Q).$$
\end{remark}

\begin{definition} A representation $E\in\Rep(Q)$ is called {\em polynomial-count}, if all its Grassmannians $\Gr_{\gamma}(E)$ are polynomial-count. It is called {\em add-polynomial-count}, if each $nE$ is polynomial-count.
\end{definition}

\begin{corollary} $nE$ is polynomial-count if and only if $\Rep_{(n,\alpha)}(Q[E])$ is polynomial-count for any $\alpha$.
\end{corollary}

If $E$ is add-polynomial-count, then $kQ[E]$ is F-polynomial-count by \eqref{eq:frext}. It follows from Lemma \ref{L:Euler}, \ref{L:Tao}, \ref{L:Frep}, and \ref{L:rep} that
\begin{theorem} \label{T:ext} $\Rep_\alpha^\mu(Q[E])$ can be explicitly counted in terms of $\Gr_\gamma(nE)$'s.
In particular, if $E$ is add-polynomial-count, then each $\Mod_\alpha^\mu(Q[E])$ is polynomial-count when it is a geometric quotient.
\end{theorem}
We will see in the last section that the assumption of being a geometric quotient can be dropped.
Polynomial-count representations of quivers include all rigid ones because of \cite[Corollary 5.2]{Fc1}, but there are many more (see Example \ref{ex:1}, \ref{ex:2n}, and \ref{ex:P2}).

\begin{question} Is there a representation, which is polynomial-count but not add-polynomial-count?
\end{question}

\section{Application to Homological Stratification} \label{S:HS}
\begin{definition} For any representation $E$, the $E$-homological stratification of $\Rep_\alpha(Q)$ is the decomposition of $\Rep_\alpha(Q)$ into (finite) disjoint union of locally closed subvarieties $\Rep_\alpha(Q;E,h)$, where
$$\Rep_\alpha(Q;E,h)=\{M\in\Rep_\alpha(Q)\mid \dim\Hom_Q(E,M)=h\}.$$
\end{definition}

By Lemma \ref{L:variety}, we know that for $n\geq 0$,
$$|\Rep_{(n,\alpha)}(Q[E])|=\sum_h |\Rep_\alpha(Q;E,h)|q^{nh}.$$
The coefficient matrix of above linear system is a non-degenerated Vandermonde-type matrix,
so we can solve all $\Rep_\alpha(Q;E,h)$. In particular,
$\Rep_\alpha(Q;E,h)$ is polynomial-count if and only if $E$ is add-polynomial-count.
We will see that the above is still true for
$$\Rep_\alpha^\mu(Q;E,h)=\{M\in\Rep_\alpha^\mu(Q)\mid \dim\Hom_Q(E,M)=h\}.$$

\begin{definition} Let $\mu=\frac{\sigma}{\theta}$ be any slope function for $Q$. The (negative) extension $\mu_-$ to $Q[E]$ with respect to an dimension vector $\alpha$ is $\frac{\sigma_-}{\theta_-}$, where $\theta_-(n,\alpha)=n+\theta(\alpha)$ and $\sigma_-(n,\alpha)=\epsilon n+\sigma_\alpha(\alpha)$ for some sufficiently small positive $\epsilon\in\mb{Q}$.
Similarly, we define the (positive) extension of $\mu_+$ to $Q[E^*]$ with respect to $\alpha$ as $\frac{\sigma_+}{\theta_+}$, where $\theta_+(\alpha,n)=\theta(\alpha)+n$, and $\sigma_+(\alpha,n)=\sigma_\alpha(\alpha)-\epsilon n$.
\end{definition}

The following lemma was proved in \cite[Theorem 5.2]{ER} for $E$ projective, but the argument goes through for any $E$.
\begin{lemma} We have the following identity in $\mc{P}_{\innerprod{Q}}$:
\begin{equation} \label{eq:framed} \Big(\sum_\beta r_\beta^\mu(Q)x^\beta\Big)\Big(\sum_\gamma r_{(n,\gamma)}^{\mu_-}(Q[E])x^\gamma\Big) = \sum_\alpha \Big(\sum_{M\in\Rep_\alpha^\mu(Q)}\frac{|\mc{O}_M||\Hom_Q(nE,M)|}{|\GL_{\alpha}||\GL_n|}\Big)x^\alpha.\end{equation}
\end{lemma}

\begin{theorem} \label{T:HS} $|\Rep_\alpha^\mu(Q;E,h)|$ can be explicitly computed from $\Gr_\gamma(nE)$. When $E$ is add-polynomial-count and $\Mod_\alpha^\mu(Q)$ is a geometric quotient, each homological strata on $\Mod_\alpha^\mu(Q)$ is polynomial-count.
\end{theorem}

\begin{proof} According to Theorem \ref{T:ext}, all $r_{(n,\alpha)}^{\mu_-}(Q[E])$'s can be computed from $\Gr_\gamma(nE)$, and so does the right hand side of \eqref{eq:framed}. Notice that
$$\sum_{M\in\Rep_\alpha^\mu(Q)}|\mc{O}_M||\Hom_Q(nE,M)|=\sum_h |\Rep_\alpha^\mu(Q;E,h)|q^{nh}.$$
We can invert the same Vandermonde-type matrix as before to solve $|\Rep_\alpha^\mu(Q;E,h)|$.
\end{proof}

\section{3-vertex Examples} \label{S:example}


Consider the $n$-arrow Kronecker quiver $K_n$ and its extension by an $(m,d)$-dimensional representation $E$. Then we can view the algebra $kK_n[E]$ as an algebra coextended from $K_m$ by a $(d,n)$-dimensional representation $E^\circ$.

It follows from Remark \ref{R:gen} that
\begin{proposition} \label{P:dual} $F^\infty(E)$ and $F_\infty(E^\circ)$ are related by
$$R(K_n)F^\infty(E)=R(K_n[E])=F_\infty(E^\circ)R(K_m).$$
In particular, if $E$ is add-polynomial-count, then so is $E^\circ$.
\end{proposition}

Let $A:=kK_m[E^*]$ be the algebra coextended from $K_m$ by a representation $E$ of dimension $\epsilon$. For any dimension vector $\alpha=(\alpha_1,\alpha_2)$ of $K_m$, there is a unique choice of weight $\sigma$ up to scalar such that $\sigma(\alpha)=0$. For the rest of this section, we always take $\mu=\frac{\sigma}{\theta}$ for different $\alpha$'s.

The first two isomorphisms below can be easily established by Lemma \ref{L:variety}.
\begin{proposition} $\Mod_{(\gamma,1)}^{\mu_+}(A)\cong \Gr_\gamma(E)$ and $\Mod_{(\gamma_1,1,1)}^{\mu_-}(A)\cong\Gr_{(\gamma_1,1)}(E)$.\\
Assume that $E$ is not too special so that $\Gr_{(n,1)}(E)$ is empty.
\begin{align*} &|\Mod_{(1,2,1)}^{\mu_-}(A)|=|\Gr_{(1,2)}(E)|+([m-1]-[\epsilon_2-1])|\Gr_{(1,1)}(E)|,\\
& |\Mod_{(2,2,1)}^{\mu_-}(A)|=|\Gr_{(2,2)}(E)|+([2m-1]-[\epsilon_2-1])|\Gr_{(2,1)}(E)|,\\
& \quad\quad\cdots\cdots\quad
\end{align*}
where $[n]$ is the quantum number.
\end{proposition}


\begin{example} \label{ex:1} Consider the quiver $$\Atwo{a}{b}$$ with relation $ab=0$. The corresponding algebra $A$ is one-point-extended from the Dynkin quiver $A_2$ by the simple $S_2$.
So $|\Rep_{(n,\alpha)}(A)|$ can be computed by Lemma \ref{L:Tao}. Note that $\Gr^\beta(nS_1)$ is just the usual Grassmannian variety $\Gr^{\beta_1}(n)$.
It follows that the quiver
$$\Atwo{a,c}{b,d}$$ with relations $ab=0,cd=0$ is polynomial-count. This algebra is extended from the Kronecker quiver $K_2$ by a decomposable non-rigid representation of dimension $(2,2)$.
\end{example}

\begin{example} \label{ex:2n} Fix $n\in\mb{N}$, we consider the quiver
$$\Atwo{a,b}{x_1\dots x_n}$$ with relation $AX=0$, where $A=(a,b)$ and $X=\sm{x_1 & x_2 & \cdots & x_{n-1}\\x_2 & x_3 & \cdots & x_{n}}$.
It is extended form $K_n$ by $E_n$ presented by
$$0\to (n-1)P_3\xrightarrow{X^T} 2P_2 \to E_n\to 0.$$
It is also coextended from $K_2$ by the exceptional $E_n^\circ$ presented by
$$0\to E_n^\circ\to nI_2 \xrightarrow{B^T} (n-1)I_1 \to 0,$$ where $B=\sm{a & b & 0 & 0 &\cdots & 0\\0 & a & b & 0 & \cdots & 0\\ \vdots & \vdots & & \ddots & \ddots & \vdots \\ 0 & 0 & \cdots & 0 & a & b}$.
Although $E_n$ is not rigid, it follows from Proposition \ref{P:dual} that it is add-polynomial-count.
$|\Gr_\gamma(E_n^\circ)|$ can be recursively computed by the cluster theory and Lemma \ref{L:GrAlg}. A closed formula was given in \cite[Theorem 4.3]{Sz}.
\begin{align*}
|\Gr_\gamma(E_n^\circ)|=\begin{cases}
1 & {\gamma=(0,0), (n+1,n)}\\
\smvar{n-\gamma_1 \\ \gamma_2-\gamma_1}\smvar{\gamma_2+1 \\ \gamma_1} & \text{ otherwise. }
\end{cases}\end{align*}
\end{example}

Now we recall \cite[Proposition 2.8]{Fc1}.
We also draw some easy consequences, which are useful for counting the Grassmannians of representations.
\begin{lemma} \label{L:GrAlg} Assume that $\dim U=\alpha_1$ and $\dim V=\alpha_2$.
\begin{equation} \sum_{\gamma_1+\gamma_2=\gamma}\innerprod{\gamma_1,\alpha_2-\gamma_2}|\Gr_{\gamma_1}(U)||\Gr_{\gamma_2}(V)|=\sum_{[W]}\frac{|\Ext_Q(U,V)_W|}{|\Ext_Q(U,V)|}|\Gr_{\gamma}(W)|.
\end{equation}
Now suppose that $\Ext_Q(U,V)=0$. Then
$$F_\bullet(U\oplus V)=\sum_{\gamma_1,\gamma_2} \innerprod{\gamma_1,\alpha_2-\gamma_2}\Gr_{\gamma_1}(U)\Gr_{\gamma_2}(V)x^{\gamma_1+\gamma_2}.$$
Hence,if both $U$ and $V$ are (add)-polynomial-count, then so is $U\oplus V$.
Moreover, if $\Ext_Q(V,U)=k^e$ and $W$ is the only non-trivial middle term of the extensions, then
$$(q^e-1)F_\bullet(W)=q^e \sum_{\gamma_1,\gamma_2} \innerprod{\gamma_2,\alpha_1-\gamma_1}\Gr_{\gamma_2}(V)\Gr_{\gamma_1}(U)x^{\gamma_1+\gamma_2}-F_\bullet(U\oplus V).$$
\end{lemma}

\begin{example} We add one arrow to Example \ref{ex:2n}:
$$\triangleone{a,b}{c}{x_1\dots x_n}$$
Then it is extended from $K_n$ by $E_n\oplus P_3$, or coextended from $K_2$ by $E_n^\circ\oplus I_1$.
Let $A_n=kK_2[(E_n^\circ\oplus I_1)^*]$.
Since $\Ext_{K_2}(E_n^\circ,I_1)=0$, we can use Lemma \ref{L:GrAlg} or compute directly $r_\alpha(A_n)$.
So we are able to find all $|\Mod_\alpha^\mu(A_n)|$. For example,
\begin{align*} & |\Mod_{(1,1,1)}^{\mu_-}(A_n)|=q^2+2q+1,\\
& |\Mod_{(1,1,1)}^{\mu_+}(A_n)|=[n]+[3]-1,\\
& |\Mod_{(2,2,1)}^{\mu_-}(A_n)|=q^4 + 2q^3 + 4q^2 + 2q + 1.
\end{align*}
However, all $\Mod_{(1,1,1)}^{\mu_-}(A_n)$ are different, they are Hirzebruch surfaces $\mb{F}_n$ \cite{Fm2}.
\end{example}


\begin{example}
Consider quiver $$\Atwo{a,b,c}{x,y,z}$$ with relation $xa+yb+zc=0$. It is coextended from $K_3$ by a rigid module presented by $0\to E\to 3I_2\xrightarrow{\sm{a & b & c}} I_1\to 0$. Similar calculation as before gives
\begin{align*}
&|\Mod_{(1,1,1)}^{\mu_\pm}(A)|=[2][3],\\
&|\Mod_{(2,1,1)}^{\mu_\pm}(A)|=|\Mod_{(1,1,2)}^{\mu_\pm}(A)|=[3],\\
&|\Mod_{(1,2,1)}^{\mu_\pm}(A)|=[3][5],\\
&|\Mod_{(2,2,1)}^{\mu_-}(A)|=|\Mod_{(1,2,2)}^{\mu_-}(A)|=[3][5](1,0,1),\\
&|\Mod_{(1,2,2)}^{\mu_-}(A)|=|\Mod_{(2,2,1)}^{\mu_+}(A)|=[3](1,1,3,3,3,1,1).
\end{align*}
Here we write the sequence $(a_n,\dots,a_1,a_0)$ for the polynomial $\sum_{i=0}^n a_i q^i$.
The first one is \cite{Fm2} a divisor $\mc{D}$ on $\mb{P}^2\times\mb{P}^2$ of bidegree $(1,1)$, or equivalently the complete flag variety $\mc{F}_3$ of $k^3$.

Now consider the deformation $E'\oplus I_2$ of $E$, where $0\to E'\to 2I_2\xrightarrow{\sm{a & b}} I_1\to 0$.
Since $\Ext_Q(I_2,E')=k$ with $E$ the only non-trivial middle term, we can compute $F_\bullet(E')$ using Lemma \ref{L:GrAlg}.
$$F_\bullet(E')=1+[2]x^{(1,0)}+[2]^2x^{(1,1)}+[2]x^{(2,1)}+x^{(0,2)}+[5]x^{(1,2)}+\smvar{5\\2}x^{(2,2)}+\cdots.$$
\begin{align*}
& |\Mod_{(1,1,1)}^{\mu_\pm}(A)|=(1,3,2,1),\\
& |\Mod_{(2,1,1)}^{\mu_\pm}(A)|=|\Mod_{(1,1,2)}^{\mu_\pm}(A)|=(2,2,1),\\
& |\Mod_{(1,2,1)}^{\mu_\pm}(A)|=[3][5],\\
& |\Mod_{(2,2,1)}^{\mu_-}(A)|=|\Mod_{(1,2,2)}^{\mu_+}(A)| =[3][5](1,0,1),\\
& |\Mod_{(1,2,2)}^{\mu_-}(A)|=|\Mod_{(2,2,1)}^{\mu_+}(A)|=[3](1,1,4,4,3,1,1).
\end{align*}
Note that the first one is irreducible and singular by Lemma \ref{L:polycount}.
\end{example}

%

\begin{example} \label{ex:P2}
Consider quiver $$\Atwo{a,b,c}{x,y,z}$$ with relation $AX=0$, where $A=\sm{0& -c& b\\c& 0& -a\\-b& a& 0}$. It is coextended from $K_3$ by $E$ presented by the following base diagram. The black dots are a basis in $E_1$; while the white dots are a basis in $E_2$. The letter on an arrow represents the identity map on the arrow of the same letter.
$$\Kthreesixthree{a}{c}{b} {a}{b}{c}{a}{b}{c}$$
It is known \cite{Fm2} that for a general representation $E_g$ of dimension $(6,3)$, $\Gr_{(1,1)}(E_g)$ is an elliptic curve. So $E_g$ is {\em not} polynomial-count. However, for this special $E$, $\Gr_{(1,1)}(E)$ is three $\mb{P}_1$'s intersecting pairwise at different points.
With a little effort \cite[Example 7.5]{Fc3}, one can show that $E$ is actually polynomial-count.
\end{example}



\section{The Universal Case: $A_2(Q)$} \label{S:A2}
Let us consider a category, which is universal in the sense that it contains all one-point extensions of $Q$ as its full subcategories. It is clearly the module category of $kA_2(Q):=kQ\otimes kA_2$, where $A_2$ is the Dynkin quiver $1\xrightarrow{} 2$.
The quiver of $kA_2(Q)$ is composed of two copies of $Q$ corresponding to two idempotents of $kA_2$, and {\em morphism arrows} connecting the same vertices in two different copies. The relations are obviously the commuting relations. By abuse of notation, we use $A_2(Q)$ to denote such a quiver with relations. So the dimension vector of $A_2(Q)$ is composed of two dimension vectors of $Q$, say $(\alpha,\beta)$. By convention, $\alpha$ correspond to the quiver sending morphism arrows.

Let $V$ be an $\alpha$-dimensional $k$-vector space. We denote by $\In_{c\cap d\hookrightarrow e}^{c_d\hookrightarrow c_e}(\alpha)$ the incidence variety
$$\{(C,D,E)\in \Gr_{c}(V)\times \Fl_{e-d,d}(V)\mid \dim(C\cap D)=c_d,\dim(C\cap E)=c_e\},$$ and
by $\Gr^{b\cap e}_d(\alpha)$ the incidence variety
$$\{(B,E)\in \Gr^{b}(V)\times \Gr^{e}(V)\mid V/B_s=B, V/E_s=E, \dim(B_s\cap E_s)=\alpha-b-e+d\}.$$

\begin{lemma} \label{L:A2Q}
$p:\Frep_{(\beta_u,\beta_d),(\gamma_u,\gamma_d)}(A_2(Q))\to \Fl_{(\beta_u,\beta_d),(\gamma_u,\gamma_d)}$ is a fibre bundle with fibre
\begin{gather}
\bigsqcup_{b,c,d,e,c_d,c_e} \In_{c\cap d\hookrightarrow e}^{c_d\hookrightarrow c_e}(\gamma_d)\times \Gr^{b\cap e}_{e-d}(\beta_u)\times \Gr^{c}(\gamma_u)\times \Gr_{b}(\beta_d)\times \GL_b\times\GL_c\times \GL_e \\
\label{eq:Sd1} \times  \prod_{a\in Q_1}\Hom(k^{c_d(ta)},k^{c_d(ha)}) \times \Hom(k^{(c_e-c_d)(ta)},k^{c_e(ha)}) \times
\Hom(k^{(c-c_e)(ta)},k^{c(ha)})\\
\label{eq:Sd2} \times  \Hom(k^{(d-c_d)(ta)},k^{d(ha)}) \times \Hom(k^{(e-d-c_e+c_d)(ta)},k^{e(ha)}) \times \Hom(k^{(\gamma_d-e-c+c_e)(ta)},k^{\gamma_d(ha)})\\
\label{eq:Tu}\times  \Hom(k^{b(ta)},k^{(b+d-e)(ta)}) \times \Hom(k^{\beta_u(ta)},k^{(\beta_u-b-d)(ha)}) \\
\label{eq:ST}\times  \Hom(k^{\gamma_u(ta)},k^{(\gamma_u-c)(ha)}) \times \Hom(k^{(\beta_d-b)(ta)},k^{\beta_d(ha)})\\
\label{eq:TS}\times  \Hom(k^{\beta_u(ta)},k^{(\gamma_u-c)(ha)})\times \Hom(k^{(\beta_d-b)(ta)},k^{\gamma_d(ha)})
\end{gather}
So $r_{(\beta_u,\beta_d),(\gamma_u,\gamma_d)}(A_2(Q)):=\frac{|\Frep_{(\beta_u,\beta_d),(\gamma_u,\gamma_d)}(A_2(Q))|}{|\GL_{(\alpha_u,\alpha_d)}|}$ is equal to
\begin{align*}\sum_{b,c,e,d_\beta,d_\gamma}t_{(b,c,d,e,c_d,c_e)}\cdot r_{\gamma_u-c}r_{\beta_d-b}\cdot r_{c_d}r_{c_e-c_d}r_{c-c_e}r_{d-c_d}r_{e-d-c_e+c_d}r_{\gamma_d-c-e+c_e}\cdot r_{\beta_u-b-d}r_{b+d-e},
\end{align*}
where $t_{(b,c,d,e,c_d,c_e)}=\frac
{(\innerprod{\beta_u,\gamma_u}\innerprod{\beta_d,\gamma_d}\cdot \innerprod{\beta_d-b,b}\innerprod{c,\gamma_u-c}\cdot \innerprod{e-d,b+d-e}\innerprod{b+d,\beta_u-b-d})^{-1}\bigl[ \begin{smallmatrix} e \\ d\end{smallmatrix}\bigr]}
{\innerprod{c-c_e,c_e-c_d}\innerprod{d-c_d,c_d}\innerprod{e-d-c_e+c_d,d+c_e-c_d}\innerprod{\gamma_d-c-e+c_e,c+e-c_e}}$, and $r_\alpha=r_\alpha(Q)$.
\end{lemma}

\begin{proof} We sketch the fibre bundle construction by a picture.
After fixing an elements in $\Fl_{(\beta_u,\beta_d),(\gamma_u,\gamma_d)}$, we need to fill in the missing part for a $(\alpha_u,\alpha_d)$-dimensional representation of $A_2(Q)$. Similar to Lemma \ref{L:Frep}, the missing part consist of a $(\gamma_u,\gamma_d)$-dimensional representation $S$, a $(\beta_u,\beta_d)$-dimensional representation $T$, and a bunch of linear maps from $T(ta)$ to $S(ha)$, as indicated below.
$$\Freptwo{\beta_u}{\beta_d}{\gamma_u}{\gamma_d}$$
The first step is to choose a configuration of image spaces of the vertical and diagonal morphism arrows.
Let $B,C$ be the images of the morphism arrows of $T$ and $S$ respectively, and $E$ be the image of the diagonal morphism arrows. Let $D$ be the image of the diagonal morphism arrows restricted on the kernel of the morphism arrows of $T$.
We assume that $\big((C,D,E),(B,E)\big)\in\In_{c\cap d\hookrightarrow e}^{c_d\hookrightarrow c_e}(\gamma_d)\times \Gr^{b\cap e}_{e-d}(\beta_u)$.

The second step is to stuff in the following order -- the lower part of $S$, the upper part of $T$, the rest part of $S$ and $T$, and other diagonal arrows. Keep in mind that there are linear maps from quotient spaces to subspaces but not vice versa.
We can easily see with the help of the following Venn diagram that they correspond to $\eqref{eq:Sd1}\times \eqref{eq:Sd2}$, $\eqref{eq:Tu}\times\GL_e$, $\eqref{eq:ST}\times\GL_c\times\GL_b$, and \eqref{eq:TS} respectively. Here inside the gibbous-shaped circles are subspaces with their dimension.

$
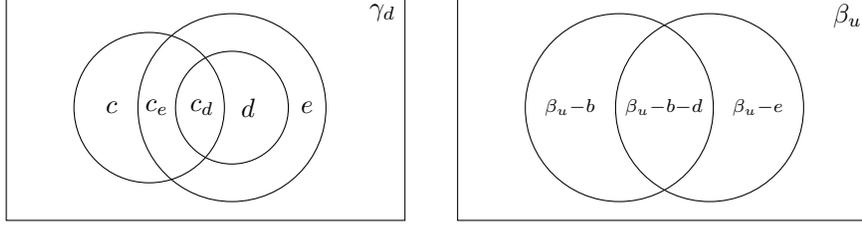
\begin{tikzpicture}
\draw (-2,-1.5) rectangle (3.3,1.5) node[below left]{$\gamma_d$};
\draw (-0.1,0) circle (1cm) node at (-0.6,0) {$c$};
\draw (1,0) circle (0.75cm) node[right] {$d$};
\draw (1,0) circle (1.25cm) node at (2,0) {$e$};
\draw node at (0,0) {$c_e$};
\draw node at (0.6,0) {$c_d$};
\draw (4,-1.5) rectangle (9.5,1.5) node[below left]{$\beta_u$};
\draw (6.15,0) circle (1.25cm) node at (5.5,0) {$_{\beta_u-b}$};
\draw (7.35,0) circle (1.25cm) node at (8,0) {$_{\beta_u-e}$};
\draw node at (6.75,0) {$_{\beta_u-b-d}$};
\end{tikzpicture}$

To compute $r_{(\beta_u,\beta_d),(\gamma_u,\gamma_d)}(A_2(Q))$, we only need to count the incidence varieties.
We use the transitive action of $\GL_\alpha$ and count the stabilizers. The following formulas are immediate.
\begin{align*} & \frac{|\Fl_{(\beta_u,\beta_d),(\gamma_u,\gamma_d)}|}{|\GL_{(\alpha_u,\alpha_d)}|}
=\frac{\innerprod{\beta_u,\gamma_u}_0^{-1}\innerprod{\beta_d,\gamma_d}_0^{-1}}{|\GL_{\gamma_u}||\GL_{\beta_u}||\GL_{\gamma_d}||\GL_{\beta_d}|},\qquad
\frac{|\In_{c\cap d\hookrightarrow e}^{c_d\hookrightarrow c_e}(\gamma_d)|}{|\GL_{\gamma_d}|}\\
&=\frac{(\innerprod{c-c_e,c_e-c_d}_0\innerprod{d-c_d,c_d}_0\innerprod{e-d-c_e+c_d,d+c_e-c_d}_0\innerprod{\gamma_d-c-e+c_e,c+e-c_e}_0)^{-1}}
{\innerprod{c_e-c_d,c_d}_0|\GL_{c_d}||\GL_{c_e-c_d}||\GL_{c-c_e}||\GL_{d-c_d}||\GL_{e-d-c_e+c_d}||\GL_{\gamma_d-c-e+c_e}|},\\
&\frac{|\Gr^{b\cap e}_{e-d}(\beta_u)|}{|\GL_{\beta_u}|}
=\frac{\innerprod{e-d,b+d-e}_0^{-1}\innerprod{b+d,\beta_u-b-d}_0^{-1}\innerprod{d,e-d}_0^{-1}}{|\GL_{d}||\GL_{\beta_u-b-d}||\GL_{b+d-e}||\GL_{e-d}|},\\
&\frac{|\Gr^{c}(\gamma_u)||\Gr_{b}(\beta_d)|}{|\GL_{\gamma_u}||\GL_{\beta_d}|}
=\frac{\innerprod{c,\gamma_u-c}_0^{-1}\innerprod{\beta_d-b,b}_0^{-1}}{|\GL_{c}||\GL_{\gamma_u-c}||\GL_{b}||\GL_{\beta_d-b}|}.
\end{align*}
Put the fibre bundle structure and these equations together, and we obtain what we desire.
\end{proof}

This result can be generalized to the $s$-step Frep varieties. So we conclude that the algebra $A_2(Q)$ is F-polynomial-count. For the $1$-step case, it suffices to set $\beta_u=\beta_d=0$.
\begin{corollary} $$r_{(\alpha,\beta)}(A_2(Q))=\sum_{\beta=\beta_1+\beta_2}\innerprod{\beta_1,\beta_2}^{-1}\innerprod{\beta_2,\alpha-\beta_2}^{-1}r_{\beta_1}(Q)r_{\beta_2}(Q)r_{\alpha-\beta_2}(Q).$$
This formula has a dual version:
$$r_{(\alpha,\beta)}(A_2(Q))=\sum_{\alpha=\alpha_1+\alpha_2}\innerprod{\alpha_1,\alpha_2}^{-1}\innerprod{\beta-\alpha_1,\alpha_1}^{-1}r_{\alpha_2}(Q)r_{\beta-\alpha_1}(Q)r_{\alpha_1}(Q).$$
\end{corollary}

\begin{remark} Alternatively, this corollary can be proved by a Hall algebra method similar to Lemma \ref{L:rep}. Consider the following identity in the algebra $H(Q)\otimes H(Q)$.
$$\Big([0]\otimes\sum_{[U]}[U]\Big)\Big(\sum_{[M],[V]}|\Epi_Q(M,V)|[M]\otimes [V]\Big) = \sum_{[M],[W]}|\Hom_Q(M,W)|[M]\otimes [W].$$
Applying the character $\oldint\otimes\oldint$ to the both sides, we see the result immediately.
\end{remark}
It follows from Lemma \ref{L:Euler}, \ref{L:Tao} and \ref{L:A2Q} that
\begin{theorem} \label{T:A2Q} If $\Mod_\alpha^\mu(A_2(Q))$ is a geometric quotient, then it has a counting polynomial, which can be explicitly computed.
\end{theorem}

We will see in the last section that the assumption of being a geometric quotient is unnecessary. This result is known \cite[Theorem 4.3]{Fc1} for some special choices of $\alpha$ and $\mu$.

\begin{example} Consider the 3-arrow Kronecker quiver $K_3$ with dimension vectors $\alpha=(3,4)$ and $\gamma=(1,3)$. Let $M$ be a general representation of dimension $\alpha$, then $M$ has no subrepresentation of dimension $(1,2)$. So the projection $\Gr_\gamma(M)\to\Gr_{1}(M_1)\cong\mb{P}^2$ is an isomorphism. We can use the algorithm in \cite[Corollary 4.4]{Fc1} to find that
\begin{align*} & |\Mod_{\alpha}^\mu(K_3)|=(1,0,1)^2(1,1,1,3,5,3,1,1,1),\\
& |\Mod_{(\gamma,\alpha)}^{\hat{\mu}}(A_2(K_3))|=[3][2]^2(1,4,2,8,5,8,2,4,1),
\end{align*}
where $\hat{\mu}=\frac{\hat{\sigma}}{\hat{\theta}}$ is the slope function constructed in \cite[Section 1]{Fc1}. Recall that $\hat{\sigma}(\gamma)=\epsilon(\gamma_1+\gamma_2)$ for some sufficiently small $\epsilon$.
Now we change $\hat{\sigma}$ to $\tilde{\sigma}(\gamma)=\epsilon\gamma_1$, then
$$|\Mod_{(\gamma,\alpha)}^{\tilde{\mu}}(A_2(K_3))|=|\mb{P}^2||\Mod_{\alpha}^\mu(K_3)|.$$
\end{example}


\begin{conjecture} If $E$ is add-polynomial-counting, then $kQ[E]\otimes kA_2$ is F-polynomial counting.
\end{conjecture}

\section{$\Delta$-analog} \label{S:delta}

Let us come back to general $A=kQ/I$. Consider the map $\int_{\Delta(\gamma)}[W]=\frac{|\Gr_\gamma(W)|}{a_W}$ as in \cite[Section 2]{Fc1}. If we apply this map to $\chi_\alpha^\mu$, we get
\begin{equation} \label{eq:intDelta} \int_{\Delta(\gamma)}\chi_\alpha^\mu=\sum_{W\in\module_\alpha^\mu(A)} {a_W}^{-1}|\Gr_\gamma(W)|.
\end{equation}
We knew from \cite[Lemma 1.6]{Fc1} that when $\Mod_\alpha^\mu(A)$ is a geometric quotient, this number is equal to $(q-1)^{-1}|\Mod_{(\gamma,\alpha)}^{\hat{\mu}}(kA_2\otimes A)|$ for some slope function $\hat{\mu}$.
To compute $\int_{\Delta(\gamma)}\chi_\alpha^\mu$, we apply $\int_{\Delta(\gamma)}$ to \eqref{eq:HallID} as before.
We define
$$\Frep_{\alpha_s\cdots\alpha_1}^\gamma(A)=\{(M,L_1,\dots,L_{s-1},S)\in \Frep_{\alpha_s\cdots\alpha_1}(A)\times \Gr_\gamma(\alpha)\mid S\subset M\},$$
then
$$\int_{\Delta(\gamma)} \chi_{\alpha_1}\cdots\chi_{\alpha_s}=|\Frep_{\alpha_1\cdots\alpha_s}^\gamma(A)|/|\GL_\alpha|.$$
Let $$\Fl_{\alpha_s\cdots\alpha_1}^{\gamma_s\cdots\gamma_1}=\{(M,L_1,\dots,L_{s-1},S)\in \Fl_{\alpha_s\cdots\alpha_1}(M)\times\Gr_\gamma(M) \mid \dim\pi_i(S\cap L_i)=\gamma_i\},$$
where $\pi_i: L_i\to L_i/L_{i-1}$ is the projection.
Then $\Frep_{\alpha_s\cdots\alpha_1}^\gamma(A)$ is stratified by the locally closed subvarieties
$$\Frep_{\alpha_s\cdots\alpha_1}^{\gamma_s\cdots\gamma_1}(A):=\Frep_{\alpha_s\cdots\alpha_1}^\gamma(A)\cap \Fl_{\alpha_s\cdots\alpha_1}^{\gamma_s\cdots\gamma_1}.$$

When $s=2$, for any decompositions $\alpha=\beta+\gamma$ and $a=b+c$, $\Fl_{\beta,\gamma}^{b,c}$ is the same as the incidence variety $$\Gr_{a\cap\gamma}^{c}(\alpha)=\{(U,V)\in \Gr_{a}(M)\times \Gr_\gamma(M)\mid \dim(U_s\cap V_s)=c\}.$$
The proof of the following lemma is similar to that of Lemma \ref{L:rep} and \ref{L:A2Q}, so we leave it for the readers.
\begin{lemma} \label{L:Frep2} $p:\Frep_{\tilde{\beta},\tilde{\gamma}}^{\tilde{b},\tilde{c}}(Q[E^*])\to \Gr_{{\tilde{a}\cap\tilde{\gamma}}}^{\tilde{c}}(\tilde{\alpha})$ is a fibre bundle with fibre \begin{align*}
& \Rep_{(\beta-b,\alpha_+)}(Q[E^*])\times \Rep_{(b,a_+)}(Q[E^*])\times \Rep_{(\gamma-c,\gamma_+)}(Q[E^*])\times \Rep_{\tilde{c}}(Q[E^*])\\
& \times \prod_{a\in Q_1}\Hom(k^{(\beta-b)(ta)},k^{(\gamma+b)(ha)})\times \Hom(k^{b(ta)},k^{c(ha)}) \times \Hom(k^{(\gamma-c)(ta)},k^{c(ha)}).
\end{align*}
So $$r_{\tilde{\beta},\tilde{\gamma}}^{\tilde{a}}(Q[E^*])=\sum_{\tilde{b}+\tilde{c}=\tilde{a}}t_{(\beta,\gamma,b,c)}\cdot r_{(\beta-b,\alpha_+)}r_{(b,a_+)}r_{(\gamma-c,\gamma_+)}r_{\tilde{c}},$$
where $t_{(\beta,\gamma,b,c)}=\frac{\smvar{\alpha_+\\(\beta-b)_+}\smvar{a_+\\b_+}\smvar{\gamma_+\\(\gamma-c)_+}|\GL_{(\gamma+b)_+}||\GL_{c_+}|}{\innerprod{\beta-b,\gamma+b}\innerprod{b,c}\innerprod{\gamma-c,c}},$ and $r_{\tilde{c}}=r_{\tilde{c}}(Q[E^*])$.
\end{lemma}

Readers can easily write out the formula for the dual case $Q[E]$. This lemma can be recursively generalized to the $s$-step case: $p:\Frep_{\alpha_s\cdots\alpha_1}^{\gamma_s\cdots\gamma_1}(A)\to \Fl_{\alpha_s\cdots\alpha_1}^{\gamma_s\cdots\gamma_1}$.
\begin{theorem} \label{T:A2ext} If $E$ is add-polynomial-count and $\Mod_\alpha^\mu(Q[E])$ is a geometric quotient, then $\sum_{M\in\Mod_\alpha^{\mu}(Q[E])}|\Gr_\gamma(M)|$ is polynomial-count for any $\gamma$.
\end{theorem}

\begin{proof} We apply $\int_{\Delta(\gamma)}$ to \eqref{eq:HallID}. Due to Lemma \ref{L:rep}, \ref{L:Frep2} and its $s$-step generalization, the right-hand side $\sum_* (-1)^{s-1} \int_{\Delta(\gamma)} \chi_{\alpha_1}\cdot\cdots\chi_{\alpha_s}$ is a polynomial in $q$.
We see from \eqref{eq:intDelta} that the left-hand side $\int_{\Delta(\gamma)} \chi_\alpha^\mu$ is equal to
$(q-1)^{-1} \sum_{W\in\Mod_\alpha^\mu(Q[E])} |\Gr_\gamma(W)|$ if $\Mod_\alpha^\mu(Q[E])$ is a geometric quotient.
As we mentioned before, according to \cite[Lemma 1.6]{Fc1} we have that
$$\sum_{W\in\Mod_\alpha^\mu(Q[E])} |\Gr_\gamma(W)| = |\Mod_{(\gamma,\alpha)}^{\hat{\mu}}(kA_2\otimes kQ[E])|$$
for some slope function $\hat{\mu}$.
Finally, our result follows from Lemma \ref{L:Euler}.
\end{proof}

As in \cite[Section 6]{Fc1}, we can also consider the $t$-step analog of $\int_{\Delta(r)}$:
$$\int_{\Delta^t(\gamma_t\cdots\gamma_1)}[W]={a_W}^{-1}|\Fl_{\gamma_t\cdots\gamma_1}(W)|.$$
Everything can be generalized to this case without any essential difficulty.

\section{$S$-analog} \label{S:S}
Finally we consider the map $\Xint S$ from the Hall algebra $H(A)$ to the formal power series algebra $\mb{Q}(q)[[\bold x]]$ in $|Q_0|$ variables as in \cite[Section 8]{Fc1}:
$$\Xint S[W]=a_W^{-1}\sum_{i=0} (-1)^{i+1} F_i(W)x^\alpha,$$
where $F_i(W)$ is the number of $i$-step filtrations of $W$. We recall from \cite[Lemma 8.3]{Fc1} that the number $\sum_{i=0} (-1)^{i+1} F_i(W)$ has a neat formula in terms of the multiplicities of simple summands of $W$.

Fix a slope function $\mu$ and a slope $\mu_0\in \mb{Q}$. Let $\module_{\mu_0}(A)$ be the abelian subcategory of all semistable representations with slope $\mu_0$, and
$\chi_{\mu_0}=\sum_{M\in\module_{\mu_0}(A)} [M]$.
Recall that a stable representation $M$ is called {\em absolutely stable} if $M\otimes_k K$ is stable for every finite field extension $k\subset K$.
Let us denote by $a_\alpha$ the number of $\alpha$-dimensional absolutely stable representations and $m_\alpha=|\Mod_\alpha^\mu(A)|$.
\begin{definition}
The absolute (resp. relative) Poincar\'{e} series of $\Rep(A)$ at $\mu_0$ is $A_{\mu_0}(A)=\sum_{\mu(\alpha)=\mu_0} a_\alpha(q)x^\alpha$ (resp. $M_{\mu_0}(A)=\sum_{\mu(\alpha)=\mu_0} m_\alpha(q)x^\alpha$).
Our convention is that the relative ones have constant term $1$, but $0$ for the absolute ones.
\end{definition}

It was proved in \cite[Theorem 4.1]{MR} that
$$\Xint S \chi_{\mu_0} = \Exp\Big(\frac{A_{\mu_0}(Q)}{1-q}\Big),$$
where $\Exp$ is the plethystic exponential in the $\lambda$-ring $\mb{Q}(q)[[\bold x]]$ \cite[Section 2]{MR}.
Moreover, it is known \cite[Theorem 8.3]{R3} that
$$M_{\mu_0}(Q)=\Exp\big(A_{\mu_0}(Q)\big).$$
Actually, for both the argument would work for any algebra not necessarily hereditary.

To compute $\Xint S \chi_{\mu_0}$, we apply $\Xint S$ to each individual $\chi_\alpha$ with $\mu(\alpha)=\mu_0$ using \eqref{eq:HallID}.
It follows from Lemma \ref{L:Frep2} and its $t$-step generalization that for the algebras $kQ[E]$ and $kA_2(Q)$ the series $A_{\mu_0}$ has all coefficients polynomials in $q$, and so are $M_{\mu_0}$.
It follows that

\begin{theorem} \label{T:final}
The assumption of being a geometric quotient in Theorem \ref{T:ext} and \ref{T:A2Q} can be dropped.
\end{theorem}

\section*{Acknowledgement} The author want to thank the anonymous referee for carefully reading the paper, and Professor Ringel for gently handeling this paper. He also thank Dr. Chari and Dr. Gierz for their kindness, for the wonderful environment of the department.

\bibliographystyle{amsplain}

\begin{thebibliography}{99}
\bibitem {ASS}  I. Assem, D. Simson, A. Skowro\'{n}ski, \textit{Elements of the representation theory of associative algebras,} London Mathematical Society Student Texts 65, Cambridge University Press, 2006.
\bibitem {CV} W. Crawley-Boevey, M. Van den Bergh, \textit{Absolutely indecomposable representations and Kac-Moody Lie algebras,} Invent. Math. 155 (2004), no. 3, 537--559.
\bibitem {DW2} H. Derksen, J. Weyman, \textit{The combinatorics of quiver representation,} Ann. Inst. Fourier (Grenoble) 61 (2011), no. 3, 1061--1131.
\bibitem {ER} J. Engel, M. Reineke, \textit{Smooth models of quiver moduli,} Math. Z. 262 (2009), no. 4, 817--848.
\bibitem {Fc1} J. Fei, \textit{Counting using Hall algebras I. Quivers,} J. Algebra 372 (2012), 542--559.
\bibitem {Fc3} J. Fei, \textit{Counting using Hall algebras III. Quivers with potentials,} arXiv:1307.2667.
\bibitem {Fm2} J. Fei, \textit{Moduli and tilting II. Extensions from quivers,} Preprint.
\bibitem {Ki} A.D. King, \textit{Moduli of representations of finite-dimensional algebras,} Quart. J. Math. Oxford Ser. (2) 45 (1994), no. 180, 515--530.
\bibitem {MR} S. Mozgovoy, M. Reineke, \textit{On the number of stable quiver representations over finite fields,} J. Pure Appl. Algebra 213 (2009), no. 4, 430--439.
\bibitem {R1} M. Reineke, \textit{The Harder-Narasimhan system in quantum groups and cohomology of quiver moduli,} Invent. Math. 152 (2003), no. 2, 349--368.
\bibitem {R2} M. Reineke, \textit{Counting rational points of quiver moduli,} Int. Math. Res. Not. 2006, Art.ID 70456, 19 pp.
\bibitem {R3} M. Reineke, \textit{Moduli of representations of quivers,} Trends in representation theory of algebras and related topics, 589--637, EMS Ser. Congr. Rep., Eur. Math. Soc., Z\"{u}rich, 2008.
\bibitem {R} C. Ringel, \textit{Hall algebras,} Topics in algebra, Part 1 (Warsaw, 1988), 433¨C447, Banach Center Publ., 26, Part 1, PWN, Warsaw, (1990).
\bibitem {Se} C.S. Seshadri, \textit{Geometric reductivity over arbitrary base,} Adv. Math. 26 (1977), 225--274.
\bibitem {Sz} C. Sz\'{a}nt\'{o} \textit{On the cardinalities of Kronecker quiver Grassmannians,}  Math. Z. 269 (2011), no. 3-4, 833--846.
\bibitem{X} J. Xiao, \textit{Drinfeld double and Ringel-Green theory of Hall algebras,} J. Algebra 190, no. 1, 100--144 (1997).

\end{thebibliography}

\end{document}